\documentclass[11pt]{amsart}
\usepackage{graphicx}
\usepackage{amssymb}
\usepackage{amsmath}

\newcommand{\R}{\mathbb R}

\newcommand{\N}{\mathbb{N}}

\newcommand{\LL}{\mathcal L}

\newtheorem{teo}{Theorem}[section]

\theoremstyle{remark}
\newtheorem{remark}[teo]{Remark}

\theoremstyle{definition}
\newtheorem{defi}[teo]{Definition}

\numberwithin{equation}{section}

\title[Numerical computation for $p(x)-$Laplace equations]{Effective numerical computation of $p(x)-$Laplace equations in 2D}
\author[A. Arag\'on, J. Fern\'andez Bonder and D. Rubio]
{Adriana Arag\'on, Juli\'an Fern\'andez Bonder and Diana Rubio}

\address{Adriana Arag\'on and Diana Rubio \hfill\break\indent
Centro de Matem\'atica Aplicada (CEDEMA), Universidad Nacional de San Mart\'in, 
Tornav\'ias Mart\'in de Irigoyen No.3100, Campus Miguelete, B1650 Villa Lynch, Provincia de Buenos Aires, Argentina.}

\email[A. Arag\'on]{\tt aaragon@unsam.edu.ar}
\email[D. Rubio]{\tt drubio@unsam.edu.ar}

\address{Juli\'an Fern\'andez Bonder \hfill\break\indent
Departamento  de Matem\'atica, FCEN -- Universidad de Buenos Aires, and\hfill\break\indent 
Instituto de C\'alculo -- CONICET\hfill\break\indent $0+\infty$ building, Ciudad Universitaria (1428), Buenos Aires, Argentina.}

\email{{\tt jfbonder@dm.uba.ar}\hfill\break\indent {\it Web page:} {\tt http://mate.dm.uba.ar/$\sim$jfbonder}}

\thanks{JFB is supported by CONICET PIP Nº 11220150100032CO and by ANPCyT, PICT 2016-1022 and is a member of CONICET}



\begin{document}

\begin{abstract}
In this article we implement a method for the computation of a nonlinear elliptic problem with nonstandard growth driven by the $p(x)-$Laplacian operator. Our implementation is based in the  {\em decomposition--coordination} method that allows us, via an iterative process, to solve in each step a linear differential equation and a nonlinear algebraic equation. Our code is implemented in {\sc MatLab} in 2 dimensions and turns out to be extremely efficient from the computational point of view.
\end{abstract}

\maketitle

\section{Introduction}
In this article we consider nonlinear elliptic problems with nonstandard growth. More precisely, we will focus on the so-called $p(x)-$Laplace operator that is defined as
$$
\Delta_{p(x)} u = \nabla \cdot \left( |\nabla u|^{p(x)-2}\nabla u\right).
$$
This operator has been used by many authors in view of its applications to several fields of sciences such as electrorheological fluids and image processing. We may cite the works \cite{CLR, R} and references therein.

From the mathematical point of view, this operator presents interesting difficulties since when the exponent $p(x)$ is not a constant, the operator is no longer homogeneous.

Recall that when $p(x)=p$ constant the $p(x)-$Laplacian becomes the well-known $p-$Laplacian operator that is well studied in the literature both from the theoretical and the computational point of view. See \cite{BL}, for instance.

In this work we present a numerical implementation of the {\em decomposition-coordination method} to solve $p(x)-$Laplacian problems of the form
\begin{equation}\label{eq.1}
\begin{cases}
-\Delta_{p(x)} u = f & \text{ in }\Omega,\\
u=g & \text{ on }\partial\Omega,
\end{cases}
\end{equation}
where $\Omega$ is a Lipschitz domain in $\R^2$ and $f\colon\Omega\to \R, g\colon \partial\Omega\to \R$ are suitable data. 

The implementation is done in {\sc MatLab} but is an arbitrary choise and our code can be easily adapted to any other programming language.

The problem of efficiently computing a solution to \eqref{eq.1} presents several difficulties due to the strong nonlinearity and more importantly due to the degenerate/singular character of the operators in points $x\in\Omega$ where $\nabla u(x)=0$ depending on $p(x)>2$ or $p(x)<2$ respectively.

As far as we know, the first rigorous analysis for numerical approximations for \eqref{eq.1} was done in \cite{DPLM}. In that article the authors studied the approximation of \eqref{eq.1} by means of the Finite Element Method and proved the convergence of the method, See also \cite{D, DPM2}.

In a subsequent paper, \cite{DPM}, the authors introduce the so-called {\em decomposition-coordination method} for problem \eqref{eq.1} that is an iterative procedure that has the advantage that in each step of the iteration one has to solve a {\em linear PDE} and a {\em nonlinear algebraic equation} which in theory makes the problem much more tractable and efficient from a computational point of view. 

The purpose of this work is to implement simple and efficiently the method developed in \cite{DPM} that allow to solve \eqref{eq.1} in dimension 2. Moreover we provide with a proof of convergence of the method under some natural assumptions.

The code implemented in MatLab can be downloaded from 
\begin{center}
{\tt https://bit.ly/3Ec1T5s}. 
\end{center}
In case the interested reader run into troubles running the program, don't hesitate in contact any of the authors.

To end this introduction, let us mention the work \cite{CZ} where the authors numerically solve \eqref{eq.1} with very different techniques that the ones presented here.

\section{Description of the method}\label{algoritmo}
In this section we give a description of the decomposition-coordination method.

This method is an iterative procedure that has the following form: First one define two vector fields $\eta_1, \nu_0\colon \Omega\to \R^2$.

Then, recursively, if we have computed $u_{n-1}$, $\eta_n$ and $\nu_{n-1}$ we compute  $u_n$, $\eta_{n+1}$ and $\nu_n$ in the following form:
\begin{enumerate}
\item Solve the {\em linear PDE}
\begin{equation}\label{eq.u}
-\Delta u_n  = \nabla\cdot(\eta _n - \nu_{n-1}) + f
\end{equation}

\item Update $\nu$ by solving the {\em algebraic nonlinear equation}
\begin{equation}\label{eq.nu}
|\nu_n|^{p-2}\nu_n + \nu_n = \eta_n + \nabla u_n.
\end{equation}

\item Finally, update $\eta$ as
\begin{equation}\label{eq.eta}
\eta_{n+1} = \eta_n + \nabla u_n - \nu_n.
\end{equation}
\end{enumerate}

In \cite{DPM} it is proved that this procedure is in fact convergent and that the sequence $\{u_n\}_{n\in\N}$ defined by \eqref{eq.u} converges to the actual solution $u$ of \eqref{eq.1}.

Have in mind that $p=p(x)$ in \eqref{eq.nu}.

Observe that in this method, in each step of the iteration, one must solve a linear PDE. This PDE is discretized using the finite element method. The matrix of the method is the same in each step.

Also observe that the nonlinear algebraic equation \eqref{eq.nu} that is a vectorial equation, can be easily solved in this form: denote $t_n=|\nu_n|$, and then $t_n \in\R$ is the solution to
$$
t_n(t_n^{p-2}+1) = r_n,
$$
where $r_n=|\eta_n + \nabla u_n|$. This can be easily solved using, for instance, the bisection method or Newton-Raphson method. Once that $t_n$ is computed, one solve
$$
\nu_n = \frac{\eta_n + \nabla u_n}{t_n^{p-1}+1}.
$$

Step 3 in the iteration is a simple evaluation.

\section{Description of the code}

We assume that the mesh $\mathcal T$ and the stiffness matrix $A$ for the Laplace operator has been generated in advance. More precisely, the information needed for the triangulation $\mathcal T$ is given by the following variables that are previously loaded in the workspace:
\begin{itemize}
\item {\tt vertex\_coordinates} is a $N_{\tt n}\times 2$ array such that the $k^{th}-$row of the matrix gives the coordinates of the $k^{th}-$vertex.

\item {\tt elem\_vertices} is a $N_{\tt t} \times 3$ array such that the $j^{th}-$row of the matrix gives the numbers of the 3 vertices defining the $j^{th}-$element.

\item {\tt dirichlet} is a $N_{\tt d}\times 1$ array such that enumerates the boundary nodes.
\end{itemize}

The stiffness matrix $A$ is an $N_{\tt n}\times N_{\tt n}$ matrix such that
$$
A(i,j) = \int_\Omega \nabla \phi_i\nabla \phi_j\, dx
$$
where $\{\phi_i\}$ is the usual picewise linear finite element basis.

Next, we create the functions that refer to our equation \eqref{eq.1}, for instance
\begin{itemize}
\item $p = {\tt @}(x,y) 1.5$

\item $f= {\tt @}(x,y) 1$

\item $g= @(x,y)0$
\end{itemize}

Here $p$ is the variable exponent in \eqref{eq.1} and $f, g$ are the source and boundary terms respectively. In this example we are computing the $p-$Laplace equation with constant exponent $p=1.5$, homogeneous Dirichlet boundary data and source $f=1$.

\subsection{Start of the algorithm}
The algorithm starts with two basic tasks:

\begin{enumerate} 
\item First define ${\tt p}$ as a $N_{\tt t}\times 1$ array as the exponent in each element that will be taken constant. For instance, what we implemented, is to compute ${\tt p}(i)$ as the value of $p$ at the barycenter of the $i^{th}-$triangle.

\item Compute ${\tt f}$ a $N_{\tt n}\times 1$ array such that 
$$
{\tt f}(k) = \int_\Omega f\phi_k\, dx.
$$
This integral is easily computed by any usual quadrature method.

\item Finally, the algorithm adjust the stiffness matrix $A$ and the source term ${\tt f}$ in order to impose the boundary data $g$.

\end{enumerate}

\subsection{The iteration}

First, define arbitrarily $\eta$ and $\nu$ as two $N_{\tt t}\times 2$ arrays. We defined them as (small) random arrays to start the iteration.
\begin{enumerate}
\item Compute the {\em divergence} of $\eta-\nu$ in the form
$$
{\tt g}(k) = \int_\Omega (\eta-\nu)\nabla \phi_k\, dx.
$$
This integral is computed by any quadrature method and recall that ${\tt g}$ is a $N_{\tt n}\times 1$ array.

\item Solve the equation $A{\tt u} = ({\tt g} + {\tt f})$.

\item Compute ${\tt gradu}$ as a $N_{\tt t}\times 2$ array such that ${\tt gradu}(i,:)$ is the gradient of ${\tt u}$ in the $i^{th}-$ triangle.

\item Solve the nonlinear algebraic equation
$$
|\bar \nu(i,:)|^{{\tt p}(i)-2} \bar \nu(i,:) + \bar \nu(i,:) = \eta(i,:) + {\tt gradu}(i,:)
$$
In order to solve this equation, for each triangle $i$, 
define the scalar quantities $r=|\eta(i,:) + {\tt gradu}(i,:)|$ and $s=p(i)$. Then solve the scalar equation
$$
x^{s-1} + x = r,\qquad x\ge 0.
$$
This problems has a unique solution, $x\in [0,r]$. We solve this problem by a bisection algorithm.

Then one takes that $\bar \nu(i,:) = (\eta(i,:) + {\tt gradu}(i,:))/(x^{s-2}+1)$.

Next one updates $\nu = \bar\nu$.

\item To finish the iteration, define $\bar\eta = \eta + {\tt gradu} - \nu$ and update $\eta=\bar\eta$.
\end{enumerate}

This process is then iterated until a stopping criteria is achieved.

\subsection{Stopping criteria}

There are several possibilities for a stopping criteria and we have not found a perfect one. In our opinion depending on the particular problem different choices of criteria worked better that others. Nevertheless, the criteria that was more robust in facing particular problems (in particular, all of the experiments presented in the next section use this criteria) is the following:
\begin{itemize}
\item Define an error tolerance $\varepsilon>0$. Then run the iteration until the relative error between two consecutive solutions is less that $\varepsilon$.
\end{itemize}

So, denoting $|{\tt u}|$ as the euclidean norm of the vector ${\tt u}$, we defined the {\em relative error} as
$$
\frac{|{\tt u} - {\tt \bar u}|}{|{\tt u}|}
$$
where ${\tt u}$ is the {\em old} solution computed in the former iteration and ${\tt \bar u}$ is the {\em new} solution computed in the last iteration.

\section{Convergence of the algorithm}

In this section we show the convergence of the algorithm described above. The convergence is a consequence of a general method described in \cite{Glowinski}.

In fact, in \cite[Chapter VI]{Glowinski}, the following general method is introduced for the approximation of minimization problems of the form
\begin{equation}\label{P}
\min_{v\in V} F(Bv) + G(v),
\end{equation}
where $V$ and $H$ are topological vector spaces, $B\in L(V,H)$ and $F\colon H\to\bar\R$ and $G\colon V\to\bar \R$ are convex, proper, lower semicontinuous functionals.

In our case, $H=\left(L^{p(x)}(\Omega)\right)^N$, $V=W_0^{1,p(x)}(\Omega)$, $Bv = \nabla v$, and
$$
F(\nu) = \int_\Omega \frac{|\nu|^{p(x)}}{p(x)}\, dx \quad \text{and}\quad G(v) = -\int_\Omega fv\, dx
$$

In order to solve the problem \eqref{P} we introduce, for any $r\ge 0$, the so-called {\em augmented Lagrangian} that is defined as
$$
\LL_r(v,\nu,\eta) := F(\nu) + G(v) + \int_\Omega \eta\cdot(\nabla v - \nu)\, dx + \frac{r}{2} \int_\Omega |\nabla v - \nu|^2\, dx.
$$
$$
\LL_r\colon W^{1,p(x)}_0(\Omega)\times \left(L^{p(x)}(\Omega)\right)^N \times \left(L^{p'(x)}(\Omega)\right)^N\to \bar \R.
$$

We need the following definition:
\begin{defi} We say that $(u^*,\nu^*,\eta^*)$ is a saddle point of $\LL_r$ if
$$
\LL_r(u^*,\nu^*,\eta)\le \LL_r(u^*,\nu^*,\eta^*)\le \LL_r(u,\nu,\eta^*)
$$
for any $\eta\in (L^{p'(x)}(\Omega))^N$ and any $(u,\nu)\in W^{1,p(x)}_0(\Omega)\times \left(L^{p(x)}(\Omega)\right)^N$.
\end{defi}

In \cite[Chapter VI]{Glowinski}, Theorem 2.1, it is proved the following
\begin{teo}
$(u^*,\nu^*,\eta^*)$ is a saddle point of $\LL_r$ if and only if $u^*$ is a solution of \eqref{P} and $\nabla u^* = \nu^*$.
\end{teo}

Therefore, to find a solution to \eqref{P} we need to find saddle points of $\LL_r$. In \cite[Chapter VI]{Glowinski} the author introduced the decomposition coordination method to approximate these saddle points.

In our case, the method consists in the following:
\begin{enumerate}
\item Let $(\nu_0,\eta_1)\in (L^{p(x)}(\Omega))^N \times (L^{p'(x)}(\Omega))^N$ be given.

\item Assume that $(\nu_{n-1}, \eta_n)$ is known and define
\begin{itemize}
\item $u_n$ such that
\begin{equation}\label{a}
\int_\Omega f(u_n-v)\, dx + \int_\Omega (\eta_n + r (\nabla u_n - \nu_{n-1}))\cdot(\nabla v - \nabla u_n)\, dx  \ge 0 
\end{equation}
for every $v\in W^{1,p(x)}_0(\Omega)$.

\item $\nu_n$ such that
\begin{equation}\label{b}
\int_\Omega \left(|\nu_n|^{p(x)-2}\nu_n - \eta_n + r(\nu_n - \nabla u_n)\right)\cdot (\nu-\nu_n)\, dx\ge 0
\end{equation}
for every $\nu\in  (L^{p(x)}(\Omega))^N$.

\item $\eta_{n+1}$ such that
\begin{equation}\label{c}
\eta_{n+1} = \eta_n + r(\nabla u_n - \nu_n).
\end{equation}
\end{itemize}

\end{enumerate}

In order to apply the convergence result of \cite{Glowinski} we need to impose the following restriction of the variable exponent $p(x)$.
\begin{teo}\label{thm.glowinski}
Assume that $p(x)\ge 2$. Let $(u^*,\nu^*,\eta^*)$ be a saddle point of $\LL_r$ and $(u_n,\nu_n,\eta_{n+1})$ be the sequence defined by the algorithm \eqref{a}--\eqref{c}. Then
\begin{align*}
u_n \to u^* &\quad \text{strongly in } H^1_0(\Omega)\\
\nu_n \to \nu^* &\quad \text{strongly in } (L^2(\Omega))^N\\
\eta_{n+1}-\eta_n \to 0 & \quad \text{strongly in } (L^2(\Omega))^N\\
\eta_n &\quad \text{is bounded in } (L^2(\Omega))^N
\end{align*}
\end{teo}

\begin{proof}
The proof is exactly Theorem 5.1 in \cite[Chapter VI]{Glowinski} in our case. The need for $p(x)\ge 2$ is in order to have the augmented Lagrangian well defined.
\end{proof}
 
It remains to see that the algorithm in Theorem \ref{thm.glowinski} is exactly the same as the one described in Section \ref{algoritmo}

\begin{teo}
The sequence defined in \eqref{a}--\eqref{c} with $r=1$ is the same as the one defined in \eqref{eq.u}--\eqref{eq.eta}.
\end{teo}

\begin{proof}
The proof is standard.

In order to see that \eqref{a} is the same as \eqref{eq.u} one take in \eqref{a} $v=u_n + w$, where $w\in C^\infty_c(\Omega)$ is arbitrary. Then, we arrive at
$$
-\int_\Omega fw\, dx + \int_\Omega (\eta_n -\nu_{n-1}) \cdot\nabla w\, dx + \int_\Omega \nabla u_n\cdot \nabla w \, dx  \ge 0.
$$
Since $w$ is arbitrary, taking now $-w$ we obtain that
$$
-\int_\Omega fw\, dx + \int_\Omega (\eta_n -\nu_{n-1}) \cdot\nabla w\, dx + \int_\Omega \nabla u_n\cdot \nabla w \, dx = 0
$$
and this is exactly \eqref{eq.u} in its weak formulation.

Next, if in \eqref{b} we take $\nu = \nu_n + \hat\nu$, it follows that
$$
\int_\Omega \left(|\nu_n|^{p(x)-2}\nu_n - \eta_n + r(\nu_n - \nabla u_n)\right)\cdot \hat\nu\, dx\ge 0.
$$
Again, changing $\hat\nu$ by $-\hat \nu$ one gets that
$$
\int_\Omega \left(|\nu_n|^{p(x)-2}\nu_n - \eta_n + r(\nu_n - \nabla u_n)\right)\cdot \hat\nu\, dx=0.
$$
and since $\hat \nu$ is arbitrary, we arrive exactly at \eqref{eq.nu}.

Finally, \eqref{c} if exactly \eqref{eq.eta}.
\end{proof}

\begin{remark}
The restriction $p(x)\ge2$ in Theorem \ref{thm.glowinski} is of technical nature and does not appear to be of relevance in the implementation. In our experiments, the algorithm behaves well for all the range $1<p(x)<\infty$. See the next section.
\end{remark}

\section{Numerical experiments}
In this section we illustrate our implementation by comparing some exact solutions to the ones computed by our method.

To begin with, let us consider $p(x,y)=p$ constant and look for two cases, one with $p>2$ where the equation \eqref{eq.1} is degenerate and other with $1<p<2$ in which equation \eqref{eq.1} is singular. Then we give an example with variable exponent $p(x,y)$.

\subsection*{Example 1}
In this first example, we consider a constant exponent $p>2$. In this case, the function
$$
u(x,y)=(x^2+y^2)^\frac{p-2}{2p-2}
$$
verifies $-\Delta_p u = 0$.

We take $\Omega = (0,1) \times (0,1)$, $p=20$ and as a boundary data the same $u$. 

We make a regular partition of $\Omega$ of $100\times 100$.

If we denote by $u_h$ the computed solution, the errors obtained are
$$
\|u-u_h\|_\infty \sim 2\times 10^{-3},\qquad \|u-u_h\|_p \sim 10^{-3}.
$$
The exact and computed solutions for this example are shown in Figure \ref{Ej1}
\begin{figure}[!h]
\begin{center}
\includegraphics[width=0.45\textwidth]{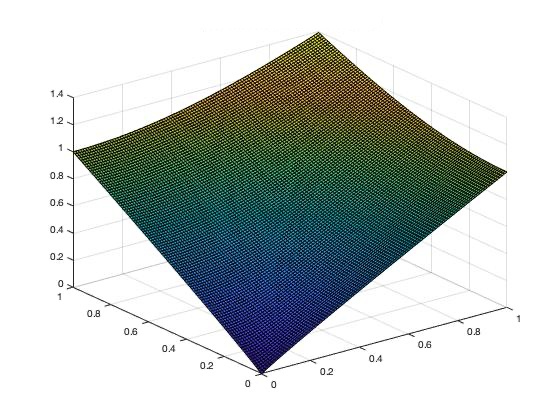}
\includegraphics[width=0.45\textwidth]{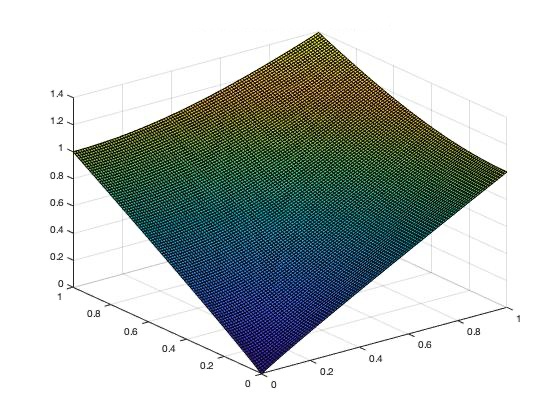}
\caption{The exact solution at the left and the computed solution at the right for Example 1} \label{Ej1}
\end{center}
\end{figure}

The time employed in the iteration was $\sim 75s$.

\subsection*{Example 2}
In this example, we consider the so-called {\em torsion problem},
$$
-\Delta_p u = 1.
$$
An exact solution to this problem is given by
$$
u(x,y) = c_p \left( 1- (x^2+y^2)^\frac{p}{2p-2}\right),
$$
where $c_p$ is a constant, $c_p = \frac{(p-1)}{p} 2^{-1/(p-1)}$.

Let us consider the domain $\Omega = (-\tfrac{1}{\sqrt{2}},\tfrac{1}{\sqrt{2}})\times (-\tfrac{1}{\sqrt{2}},\tfrac{1}{\sqrt{2}})$  and constant $p$, $p=1.1$. As boundary data we consider the same $u$ and again we make a regular partition of the domain $\Omega$ of $100\times 100$.

As in the previous example, we call $u_h$ the computed solution and the errors obtained are
$$
\|u-u_h\|_\infty \sim 2\times 10^{-4},\qquad \|u-u_h\|_p \sim 3\times 10^{-4}.
$$
The exact and computed solutions for this example are shown in Figure \ref{Ej2}.

\begin{figure}[!h]
\begin{center}
\includegraphics[width=0.45\textwidth]{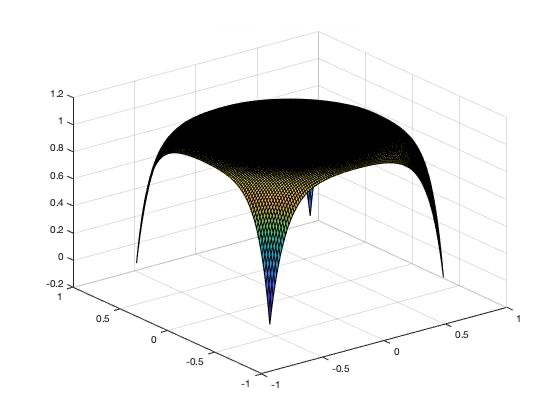}
\includegraphics[width=0.45\textwidth]{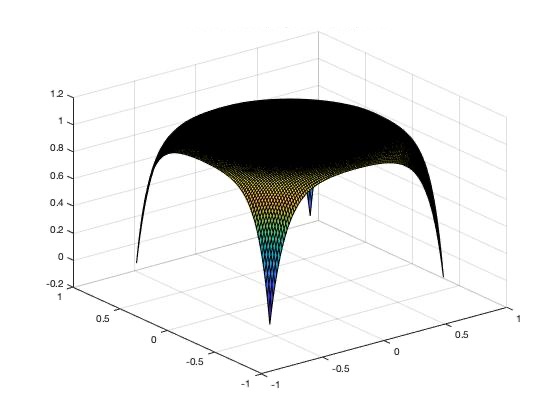}
\caption{The exact solution at the left and the computed solution at the right for Example 2} \label{Ej2}
\end{center}
\end{figure}

The time employed in the iteration was $\sim 80s$.

\subsection*{Example 3}
In this example we consider a variable exponent $p$
$$
p(x,y)= 1+\left(\frac{x+y}{2}+2\right)^{-1} \quad \text{and}\quad  u(x,y)= \sqrt{2} e^2 (e^{\frac12 (x+y)}-1).
$$
This function $u$ is a solution to $-\Delta_{p(x,y)} u = 0$.

Consider now the domain $\Omega=(-1,1)\times (-1,1)$ with a regular partition of $100\times 100$ and take as boundary data the same $u$.

Again, calling $u_h$ to the computed solution, the errors obtained are
$$
\|u-u_h\|_\infty \sim 1.5\times 10^{-5},\qquad \|u-u_h\|_p \sim 2.7\times 10^{-4}.
$$
The exact and computed solutions for this example are shown in Figure \ref{Ej3} and the time employed in the iteration was $\sim 81s$.

\begin{figure}[!h]
\begin{center}
\includegraphics[width=0.45\textwidth]{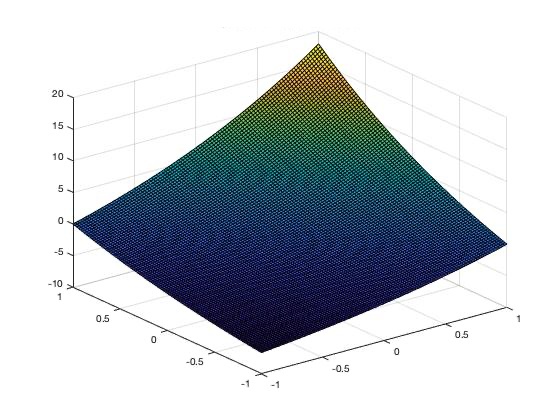}
\includegraphics[width=0.45\textwidth]{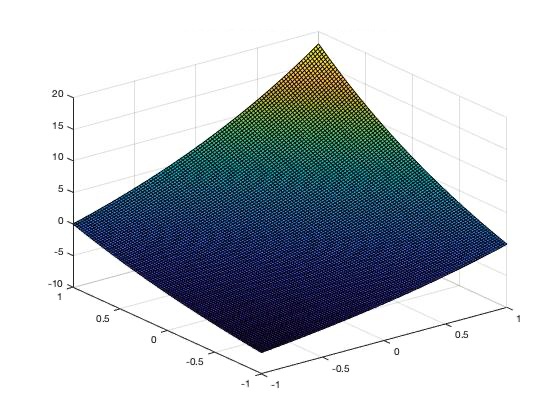}
\caption{The exact solution at the left and the computed solution at the right for Example 3} \label{Ej3}
\end{center}
\end{figure}

\subsection*{Example 4 - The torsion problem}
In this example, we work with the {\em torsion problem}, with homogeneous Dirichlet boundary data and constant exponent $p$. This is
$$
\begin{cases}
-\Delta_p u = 1 & \text{ in }\Omega\\ 
u=0  & \text{ on }\partial\Omega.
\end{cases}
$$

An exact solution to this problem is not available. We consider $\Omega= (-1,1)\times (-1,1)$ and a regular partition of $100\times 100$. We obtain the following computed solutions for different values of the exponent $p$. Some of theses solutions are shown in Figure \ref{Ej4}
\begin{figure}[!h]
\begin{center}
\includegraphics[width=0.25\textwidth]{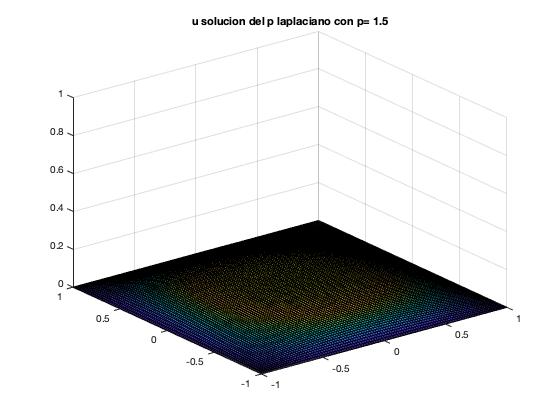}
\includegraphics[width=0.25\textwidth]{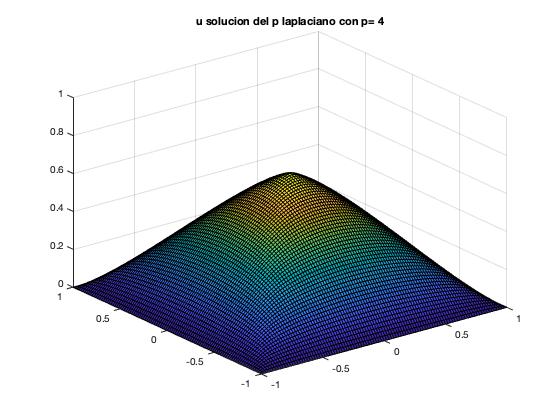}
\includegraphics[width=0.25\textwidth]{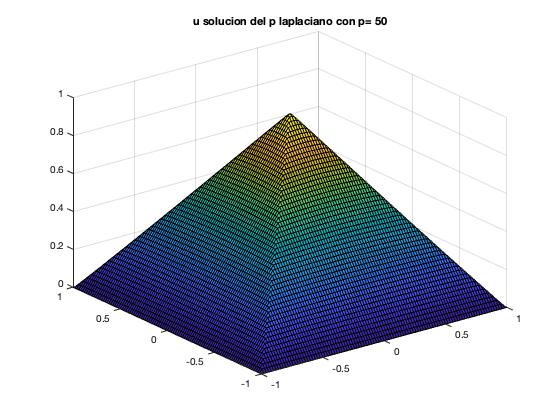}
\caption{To the left, the computed solution with $p=1.15$, in the center the computed solution with $p=4$ and to the right the computed solution with $p=50$ for the $p-$torsion problem in Example 4} \label{Ej4}
\end{center}
\end{figure}

Let us observe that for small values of $p$ the solution is {\em flat} and as the value of $p$ increases, the solution is getting closer to the {\em pyramid} with maximum at the center of the domain.

\subsection*{Example 5 - variable exponent}
Finally, we analyze the torsion problem for variable exponents
$$
\begin{cases}
-\Delta_{p(x,y)} u=1 & \text{ in }\Omega,\\
u=0  & \text{ on }\partial\Omega,
\end{cases}
$$
We consider a {\em discontiuous} exponent,
$$p(x,y) =
\begin{cases}
1.2 & x \leq 0,\\
4 & x>0
\end{cases}
$$
and a rectangular domain $\Omega=  [-2,2]\times [-1,1]$.

Again, an exact solution is not available.

Our method can handle also this case and the computed solution is shown in Figure \ref{Ej5} 
\begin{figure}
\begin{center}
\includegraphics[width=0.45\textwidth]{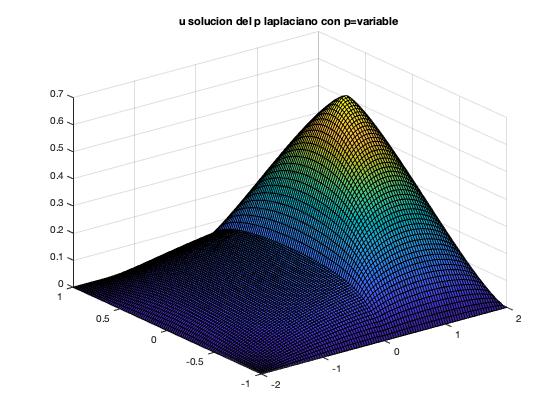}
\end{center}
\caption{The torsion problem with variable and discontinuous exponent in Example 5}\label{Ej5}
\end{figure}

As expected, the solution is {\em flat} in the region $\{x \leq 0\}$ where the exponent is small ($p(x,y)=1.2$) and has a {\em pyramid shape} in the region $\{ x>0\}$ where the exponent is large ($p(x,y)=4$).

The time employed in the iteration was $\sim 52s$.

\section{Conclusions}

We implemented a Finite Element based algorithm for the computation of solutions of equations with nonstandard growth of $p(x)-$Laplacian type. This algorithm uses the decomposition-coordination method that has the advantage that it is an iterative scheme that in each step of the iteration one has to solve a linear PDE and a nonlinear algebraic equation. Our numerical examples show that this method is extremely efficient for these cases.

\section*{Acknowledgements}

The authors want to thank I. Ojea for some useful discussions and the help provided in an early version of the code.

This research is partially supported by CONICET PIP Nº 11220150100032CO and by ANPCyT, PICT 2016-1022.


\end{document}